\documentclass{amsart}

\usepackage{amsmath, amsthm, amssymb}
\usepackage[alphabetic]{amsrefs}
\usepackage{color}
\usepackage{hyperref}
 
\newtheorem{theorem}{Theorem}[subsection]
\newtheorem{corollary}{Corollary}[subsection]

\newtheorem{proposition}{Proposition}[subsection]
 
\theoremstyle{definition}
\newtheorem{definition}{Definition}[subsection]

\newcommand\C{{\mathbb C}}
\newcommand\A{{\mathbb A}}
\newcommand\Z{{\mathbb Z}}

\newcommand{\LND}{\mathrm{LND}}
\newcommand{\Ker}{\mathrm{Ker}}
\newcommand{\Spec}{\mathrm{Spec}}

\newcommand\Vcal{{\mathcal V}}

\newcommand\Ucal{{\mathcal U}}

\title{Isomorphisms between cylinders over Danielewski surfaces}
\author{Lucy Moser-Jauslin}
\address{Institut de Math\'{e}matiques de Bourgogne, UMR 5584 CNRS, Universit\'{e} Bourgogne Franche-Comt\'{e}, F-21000 Dijon, France}
\email{lucy.moser-jauslin@u-bourgogne.fr}
\author{Pierre-Marie Poloni}
\address{Universit\"at Basel, Departement Mathematik und Informatik, Spiegelgasse 1, CH--4051 Basel, Switzerland}
\email{pierre-marie.poloni@unibas.ch}
\thanks{The first author received  support from the French \textquotedblleft Investissements d\textquoteright Avenir \textquotedblright program,  project ISITE-BFC (contract ANR-lS-IDEX-OOOB).
The second author acknowledges support by the Swiss National Science Foundation Grant \textquotedblleft Curves in the spaces\textquotedblright  200021--169508.}

\begin{document}

\begin{abstract}
A special Danielewski surface is an affine surface which is the total space of a principal $(\C,+)$-bundle over an affine line with a multiple origin.   Using a fiber product trick introduced by Danielewski, it is known that cylinders over two such surfaces are always isomorphic provided that both bases have the same number of origins. The goal of this note is to give an explicit method to find isomorphisms between cylinders over special Danielewski surfaces. The method is based on the construction of appropriate locally nilpotent derivations.
\end{abstract}

    \maketitle
    
\section{Introduction}

In 1989, Danielewski exhibited a family of pairwise non-isomorphic complex affine rational surfaces $Y_n$, $n\geq1$, such that the cylinders $Y_n\times \A^1$ are all isomorphic. The surface $Y_n$ is defined to be the hypersurface in $\A^3$ defined by $x^ny=z^2-1$ for every positive integer $n$. Since this result, several authors have generalized Danielewski's construction and have introduced the notion of  Danielewski surfaces.  These are certain affine surfaces which can be realized as the total space of an $\A^1$-fibration over the affine line. Special Danielewski surfaces  have the stronger property of being the total space of a principal $(\C,+)$-bundle over an affine line with a multiple origin. They were introduced in  \cite {DuPo} and are those Danielewski surfaces for which Danielewski's original argument can be used to find isomorphic cylinders.  However, the proof of these isomorphisms is not constructive. 

The main result of this article is to give a method to find  explicit isomorphisms of these cylinders. More precisely, the theorem~\ref{main-thm} produces, for every special Danielewski surface, an isomorphism between the cylinder over this surface and the cylinder over a classical Danielewski surface defined by an equation of the form  $xy=P(z)$ in $\A^3$. This involves the construction of an appropriate $(\C,+)$-action on the cylinder of one surface whose quotient gives the other  Danielewski surface. 

As a corollary, one gets  explicit  embeddings of all special Danielewski surfaces as complete intersections in $\A^4$.

\medskip
The paper is organized as follows.

In the section two, we recall the construction of Danielewski surfaces and some of their important properties, due to Fieseler and Dubouloz.  Then in the following section we introduce three particular families of special Danielewski surfaces which are later used as examples for the main result in section 4.  Two of these families are constructed as hypersurfaces, whereas for the last family, we do not know if they are realizable as hypersurfaces or not. In section four, we establish the theorem~\ref{main-thm}, which shows how to construct an isomorphism between the cylinders of any two special Danielewski surfaces. Finally, in section 5, we apply this result to the families of surfaces described in section 3. In particular, we obtain in the proposition~\ref{prop-classical-DS} a very simple explicit isomorphism between the cylinders of any two classical Danielewski surfaces whose respective equations are of the form $x^ny=P(z)$ and $x^my=Q(z)$.

\medskip
{\bf Acknowledgments.} Part of this work was done during the first joint meeting Brazil-France in Mathematics. The second-named author gratefully acknowledges financial support from   the R\'eseau Franco-Br\'esilien  de Math\'ematiques (RFBM).

\section{Danielewski surfaces after Danielewski,  Fieseler and Dubouloz}

In this section, we introduce some notations and summarize basic facts about Danielewski surfaces due to Fieseler \cite{Fi} and Dubouloz \cite{Du} (see also \cite{DuPo}). 

\subsection{Construction of Danielewski surfaces}

\begin{definition}
A Danielewski surface is a smooth complex affine surface $S$
equipped with an $\A^1$-fibration $\pi\colon  S\to\A^1=\Spec(\C[x])$  that restricts to a trivial $\A^1$-bundle over $\A^1_*=\Spec(\C[x,x^{-1}])$  such that the exceptional fiber $\pi^{-1}(0)$ is reduced and consists of a disjoint union 
\[\pi^{-1}(0)=\coprod_{i=1}^d \ell_i\]
of $d\geq2$ curves, $\ell_1,\ldots,\ell_d$, all isomorphic to the affine line.
\end{definition}

For every $1\leq i\leq d$, we denote by $\Ucal_i\subset S$ the open subvariety of $S$  defined by  
\[\Ucal_i=S\smallsetminus\coprod_{j\neq i}\ell_j\subset S.\]
Since every $\Ucal_i$  is isomorphic to the affine plane $\A^2$,  every Danielewski surface can be constructed by gluing together $d\geq2$ copies of $\A^2$ along $\A^1_*\times\A^1$. More precisely, every Danielewski surface is isomorphic to a variety $S(d,\boldsymbol{\sigma})$ defined as follows.

\begin{definition}
Let $d\geq2$ be an integer and let    
\[\boldsymbol{\sigma}=\big((n_1,\sigma_1(x)),\ldots,(n_d,\sigma_d(x))\big)\in(\Z_{>0}\times\C[x])^d\] be a sequence such that the polynomials $\sigma_i$ are distinct and satisfy that  $\deg(\sigma_i(x))<n_i$ for all $1\leq i\leq d$. We denote by $S(d,\boldsymbol{\sigma})$ the surface obtained by gluing together $d$  copies $\Ucal_i=\Spec(\C[x,u_i])$ of $\A^2$ along the open subsets 
\[\Ucal_i^*=\Spec(\C[x,x^{-1},u_i])\simeq\C^*\times\C\]  via the transition functions  
\begin{align*}
\Ucal_i^*  &\to\Ucal_j^* \\
(x,u_i)&\mapsto (x,x^{n_i-n_j}u_i+\frac{\sigma_i(x)-\sigma_j(x)}{x^{n_j}}).
\end{align*}
\end{definition}
By \cite[Proposition 1.4]{Fi}, every such surface $S=S(d,\boldsymbol{\sigma})$  is  affine. Moreover, the inclusion $\C[x]\hookrightarrow\C[S]$ defines an $\A^1$-fibration $\pi\colon S\to\A^1$ such that $\pi^{-1}(\A^1_*)\simeq\A^1_*\times\A^1$ and such that the  unique special fiber $\pi^{-1}(0)$   consists of a disjoint union of $d$ reduced copies of $\A^1$. This shows that  $S(d,\boldsymbol{\sigma})$ is indeed a Danielewski surface.

By construction, every Danielewski surface $S=S(d,\boldsymbol{\sigma})$  is canonically equipped with a regular function $u\in\C[S]$ whose restrictions to each of the open subsets $\Ucal_i$ are given by
\[u|_{\Ucal_i} = x^{n_i} u_i +\sigma_i(x) \in \C[x, u_i].\]
Note that $u$ restricts to a coordinate function on every general fiber of $\pi=\textrm{pr}_x\colon S\to\A^1$, but not on the exceptional fiber $\pi^{-1}(0)$.

\subsection{Additive group actions and isomorphic cylinders.}

Every Danielewski surface $S=S(d,\boldsymbol{\sigma})$  is canonically equipped with a regular $(\C,+)$-action $\delta:\C\times S\to S$ defined on each chart $\Ucal_i$   by 
\[\delta(\lambda,(x,u_i))=(x,u_i+\lambda x^{n-n_i}),\]
where $n=\max\{n_i\mid 1\leq i\leq d\}$. 
Algebraically,  the action $\delta$ corresponds to the locally nilpotent derivation $D\in\textrm{LND}(\C[S])$ that is defined by $D(x)=0$ and $D(u_i)=x^{n-n_i}$. Note that $D(u)=x^n$.

An important property of Danielewski surfaces is the fact that the map $\pi=\textrm{pr}_x\colon  S\to\A^1$ factors through a locally trivial fiber bundle $S\to Z(d)$ over the affine line $Z(d)$ with a $d$-fold origin, where the preimages of the $d$ origins are the $d$ affine lines $\ell_i$. In the case when the $(\C,+)$-action $\delta$ is  free, we have moreover that $S$ is  the total space of a $(\C,+)$-principal bundle over $Z(d)$. Recall (see \cite[Section 2.10]{DuPo}) that $\delta$   is free if and only if all $n_i$ are equal to each other, i.e.~ if and only if $n_i=n$ for all $1\leq i\leq d$. The latter condition is equivalent to the fact that the canonical class of  $S$ is trivial.  These Danielewski surfaces were called \emph{special} in \cite{DuPo}. 

Danielewski's fiber product trick goes then as follows. Take two Danielewski surfaces, say $S$ and $S'$, that are $(\C,+)$-principal bundles over the same $Z(d)$ and consider their fiber product $S\times_{Z(d)}S'$. Since every  $(\C,+)$-principal bundle over an affine base is trivial, we get at once that
\[S\times\A^1\simeq S\times_{Z(d)}S'\simeq S'\times\A^1,\]
hence that the cylinders over $S$ and $S'$ are isomorphic to each other.

\section{Examples of special Danielewski surfaces}

\subsection{Classical Danielewski surfaces.}
These surfaces are the ones  originally considered by Danielewski. They are defined as the hypersurfaces $W_{n,P}$ in $\A^3$ of equation 
\[W_{n,P}\colon x^ny=P(z),\]
where $n\geq1$ is a positive integer and where $P(z)=\prod_{i=1}^d(z-r_i)\in\C[z]$ is a polynomial  with  $d\geq2$ simple roots. 
 
Together with the restriction of the first projection $\pi=\textrm{pr}_x\colon W_{n,P}\to \A^1$, every such surface   defines a Danielewski surface. The special fiber $\pi^{-1}(0)$ is the union of the lines $\ell_1,\ldots,\ell_d$ given by
\[\A^1\simeq\ell_i=\{(0,y,r_i)\mid y\in\C\}\subset W_{n,P}.\]

Every open set $\Ucal_i=W_{n,P}\smallsetminus\coprod_{j\neq i}\ell_j$ is isomorphic to $\A^2$ and we have the isomorphisms  
\[\varphi_i\colon \Ucal_i\xrightarrow{\sim}\A^2, (x,y,z)\mapsto(x,u_i), \text{ where } u_i=\frac{z-r_i}{x^n}=\frac{y}{\prod_{j\neq i}(z-r_j)}\in\C[\Ucal_i].\]

\subsection{Danielewski hypersurfaces.} The hypersurfaces in $\A^3$ that are defined by an equation of the form 
\[H_{n,Q}\colon x^ny=Q(x,z),\]
where $n\geq1$ and where $Q(x,z)\in\C[x,z]$ is such that $\deg(Q(0,z))\geq2$ are called \emph{Danielewski hypersurfaces}. If moreover  the polynomial $Q(0,z)\in\C[z]$ has $d\geq2$ simple roots, say $r_1,\ldots,r_d$, then $\pi=\text{pr}_x\colon H_{n,Q}\to\A^1$ defines a Danielewski surface. Its special fiber is the union of the lines $\ell_1,\ldots,\ell_d$ given by
\[\A^1\simeq\ell_i=\{(0,y,r_i)\mid y\in\C\}\subset H_{n,Q}.\] 

Furthermore, there exist unique polynomials $\sigma_1(x),\ldots,\sigma_d(x)\in\C[x]$  of degree strictly smaller than $n$   such that $\sigma_i(0)=r_i$ and such that the congruences 
\[Q(x,\sigma_i(x))\equiv 0 \mod(x^n)\]
hold for all $1\leq i\leq d$.  Then, every open set $\Ucal_i=H_{n,Q}\smallsetminus\coprod_{j\neq i}\ell_j$ is isomorphic to $\A^2$ and we have the isomorphisms  
\[\varphi_i\colon \Ucal_i\xrightarrow{\sim}\A^2, (x,y,z)\mapsto(x,u_i), \text{ where } u_i=\frac{z-\sigma_i(x)}{x^n}.\]
(See \cite{DuPo} for the details.)

\subsection{Iterated Danielewski hypersurfaces}\label{Section:construction-iterated-Danielewski}

Introduced by Alhajjar \cite{Al}, iterated Danielewski hypersurfaces  are the hypersurfaces in $\A^3$ that are defined by an equation of the form
\[H_{n,Q,m,R}\colon x^mz=R(x,x^ny-Q(x,z)),\]
where $n,m\geq1$ and where $Q(x,t),R(x,t)\in\C[x,t]$. If the polynomial $R(0,-Q(0,t))$ in $\C[t]$ has only $d\geq2$  simple roots, then $H_{n,Q,m,R}$ is a Danielewski surface.

We discuss now a specific example in details. Consider the surface  $H\subset\A^3$ defined by  
\[H\colon \{xz=(xy+z^2)^2-1\}.\] 
The special fiber of $\pi=\textrm{pr}_x\colon H\to\A^1$ consists of the four lines $\ell_1,\ldots,\ell_4$ given by 
\[\ell_i=\{(0,y,\varepsilon^i)\mid y\in\C\},\]
where $\varepsilon=\boldsymbol{i}$ denotes a primitive fourth root of the unity. 

Letting $u=xy+z^2$ and $\sigma_i(x)=\varepsilon^{2i}+\frac{\varepsilon^{-i}}{2}x$ for all $1\leq i\leq 4$,  it follows  that the open set $\Ucal_i=H\smallsetminus\coprod_{j\neq i}\ell_j$ is isomorphic to $\A^2$ and one claims that  the map  
\[\varphi_i\colon \Ucal_i\xrightarrow{\sim}\A^2, (x,y,z)\mapsto(x,u_i), \text{ where } u_i=\frac{u-\sigma_i(x)}{x^2}\]
is an isomorphism.
\begin{proof}
First, we remark that the rational functions
\[\alpha_i=\frac{u-\varepsilon^{2i}}{x}=\frac{z}{u+\varepsilon^{2i}}\quad \text{ and }\quad \beta_i=\frac{z-\varepsilon^i}{x}=\frac{z-xy^2-2yz^2}{\prod_{j\in\{1,\ldots,4\}\smallsetminus\{i\}}(z-\varepsilon^j)}\]
are  regular   on $\Ucal_i$.   It follows  that $u_i$ is also an element of $\C[\Ucal_i]$, since one easily checks that 
\[\beta_i-\left(\alpha_i\right)^2=2\varepsilon^{2i}u_i.\]
Finally, the fact that $\varphi_i$ is an isomorphism follows from the following identities in $\C[\Ucal_i]$.
\begin{align*}
\alpha_i &= xu_i+\frac{\varepsilon^{-i}}{2}\\
\beta_i &=2\varepsilon^{2i}u_i+\left(\alpha_i\right)^2\\
z&=\varepsilon^i+x\beta_i\\
y&= \frac{u-z^2}{x}=\frac{u-(\varepsilon^i+x\beta_i)^2}{x}=\alpha_i-2\varepsilon^i\beta_i-x(\beta_i)^2.
\end{align*}
\end{proof}

\subsection{Double Danielewski surfaces}\label{Section:construction-double-Danielewski}

In \cite{GuSe}, Gupta and Sen studied some surfaces defined by two equations in $\A^4$ of the form
\[S\colon \{x^ny=Q(x,z) \text{ and } x^mt=R(x,z,y)\},\]
where $n,m\geq1$ and where $Q(x,z)\in\C[x,z]$ and $R(x,z,y)\in\C[x,z,y]$. They call them \emph{double Danielewski surfaces}. Indeed, if $Q(0,z)\in\C[z]$ has $d\geq2$ simple roots, say $r_1,\ldots,r_{d}$, and if every polynomial $R(0,r_i,y)\in\C[y]$ also has only simple roots, then $S$ is a Danielewski surface together with the first projection $\textrm{pr}_x\colon S\to\A^1$.   

Let us study here a specific example in details, namely the surface $D\subset\A^4$ defined by  
\[D\colon \{xy=z^2-1 \text{ and } xt=y^2-1\}.\] 
 It is a Danielewski surface, its special fiber  $\textrm{pr}_x^{-1}(0)$ consisting of the four lines given by  $\{(0,\pm1,\pm1,t)\mid t\in\C\}\subset D$. Let us introduce the following notation. For every pair $(i,j)\in\{-1,1\}\times\{-1,1\}$, we let 
\[\ell_{i j}=\{(0,i,j,t)\mid t\in\C\}\]
and 
\[\sigma_{i j}(x)=j+\frac{ij}{2}x.\]
Then, every open set $\Ucal_{i j}=D\smallsetminus\coprod_{(i',j')\neq (i,j)}\ell_{i' j'}$ is isomorphic to $\A^2$ and one claims that the map  
\[\varphi_{i j}\colon \Ucal_{i j}\xrightarrow{\sim}\A^2, (x,y,z,t)\mapsto(x,u_{i j}) \text{ where } u_{i j}=\frac{z-\sigma_{i j}(x)}{x^2}\]
is an isomorphism.

\begin{proof}
First, remark that $\alpha_j=\frac{z-j}{x}=\frac{y}{z+j}$ and $\beta_i=\frac{y-i}{x}=\frac{t}{y+i}$ are regular functions on  $\Ucal_{i j}$. Hence, it is straightforward to check that $u_{i j}$ is a regular function on $\Ucal_{i j}$, since
\[\beta_i-\left(\alpha_j\right)^2=2ju_{i j}.\]
The fact that $\varphi_{i j}$ is an isomorphism follows from the following identities in $\C[\Ucal_{i j}]$.
\begin{align*}
\alpha_j &= xu_{i j}+\frac{ij}{2}\\
\beta_i &=2ju_{i j}+\left(\alpha_j\right)^2\\
z&=j+x\alpha_j\\
y&=i+x\beta_i\\
t&= (y+i)\beta_i.
\end{align*}
\end{proof}

\section{Isomorphisms between  cylinders}\label{Section:plan of the construction}

In this section, we fix  an integer $d\geq 2$ and denote by $Z(d)$ the affine line with $d$ origins. Let  
  $P(z)=\prod_{i=1}^d(z-r_i)\in\C[z]$  be a polynomial with  simple roots. We will explain how to construct, given a special Danielewski surface $S$ which is a principal bundle over $Z(d)$, an isomorphism between its cylinder $S\times\A^1$ and the cylinder $W\times\A^1$ over the classical Danielewski surface
\[W=W_{1,P}\colon \{xy=P(z)=\prod_{i=1}^d(z-r_i)\} \text { in } \A^3.\]

Recall that a special Danielewski surface $S$ is constructed from a data set consisting of  a positive integer  $n\geq1$ and of distinct polynomial $\sigma_1(x),\ldots,\sigma_d(x)\in\C[x]$  of degree strictly smaller than $n$. More precisely,   $S$ is obtained by gluing $d$ copies, $\Ucal_1,\ldots,\Ucal_d$, of $\A^2=\Spec(\C[x,u_i])$ along $\A^1_*\times\A^1$ by means of the transition functions 
\[(x,u_i)\mapsto (x,u_i+\frac{\sigma_i(x)-\sigma_j(x)}{x^{n}}).\]
We also recall that the inclusion $\C[x]\hookrightarrow S$ defines an $\A^1$-fibration $\pi\colon S\to\A^1$ with a unique special fiber $\pi^{-1}(0)=\coprod_{i=1}^d \ell_i$ consisting of $d$ disjoint reduced copies of $\A^1$, and that we can define, by considering the regular function $u\in\C[S]$ whose restrictions on the open sets $\Ucal_i$ are given by \[u|_{\Ucal_i} = x^{n} u_i +\sigma_i(x) \in \C[x, u_i],\]
  the canonical locally nilpotent derivation $D\in\LND(\C[S])$ by setting $D(x)=0$ and $D(u)=x^n$.

Following Danielewski's original argument, we  consider the fiber product $S\times_{Z(d)}W$,  which we denote by $V$. We will use the following notations. We identify the ring of regular functions on $W$ with its canonical image in the ring of regular functions on $V$ and write 
\[\C[W]=\C[x,y,z]\subset\C[V], \quad\text{ where } xy=P(z).\]

Similarly, we identify $\C[S]$ as a subring of $\C[V]$.
Then, $V$ can be naturally seen as being obtained by gluing $d$ copies  $\Vcal_i=\Spec[x,u_i,z_i]$ of $\A^3$ where $z_i=(z-r_i)/x$. The open subvarieties $\Vcal_i$ are glued together along $\A^1_*\times\A^2$ via the transition functions
\[(x,u_i,z_i)\mapsto(x,u_i+\frac{\sigma_i(x)-\sigma_j(x)}{x^{n}},z_i+\frac{r_i-r_j}{x}).\] 
In particular, we have that the regular functions $u\in\C[S]\subset\C[V]$ and $z\in\C[W]\subset\C[V]$ satisfy that
\[u|_{\Vcal_i} = x^{n} u_i +\sigma_i(x) \in \C[\Vcal_i]=\C[x,u_i,z_i]\]
and 
\[z|_{\Vcal_i} = x z_i +r_i \in \C[\Vcal_i]=\C[x,u_i,z_i]\]
for all $1\leq i\leq d$.

\medskip
{\bf Plan of the construction.} 
Our construction of an isomorphism between $S\times\A^1$ and $W\times\A^1$  consists of three steps. We first find a regular function $\alpha\in \C[V]$ on the fiber product $V=S\times_{Z(d)}W$ such that 
\[\C[V]=\C[S][\alpha]\simeq\C[S\times\A^1].\]
We then use this equality to extend the canonical derivation on $\C[S]$ to a locally nilpotent derivation $\tilde{D}$ on  $\C[V]$ in such a way that \[\Ker(\tilde{D})=\C[x,y,z]=\C[W]\subset\C[V].\] Finally, in the last step, we construct an element $s\in\C[V]$ which is a slice for $\tilde{D}$. This gives 
\[\C[S\times\A^1]\simeq\C[S][\alpha]=\C[V]=\Ker[\tilde{D}][s]=\C[W][s]\simeq\C[W\times\A^1]\]   
hence the desired isomorphism between $S\times\A^1$ and $W\times\A^1$.

\medskip
{\bf Step 1.} Since $\ell_1,\ldots,\ell_d$ are disjoint closed subvarieties of the affine variety $S$, there exists a regular function $f\in\C[S]$ such that 
\begin{equation}\tag{$\star$}\label{equa*}
f|_{\ell_i} = r_i \quad\text{for all } 1\leq i\leq d.
\end{equation}
In other words, we can choose a function $f\in\C[V]$ such that
\[f|_{\Vcal_i}=r_i+x\tilde{f}_i \quad \text{for all } 1\leq i\leq d,\]
where $\tilde{f}_i$ is some element in $\C[x,u_i]\subset\C[\Vcal_i]$.
Therefore, the rational function   \[\alpha=\frac{z-f}{x}\]
is in fact a regular function on $V$, since 
\[(z-f)|_{\Vcal_i}=xz_i+r_i-r_i-x\tilde{f}_i=x(z_i-\tilde{f}_i)\]
is divisible by $x$ for all $i$. 

It is then straightforward to check that $z=f+x\alpha$ and $y=x^{-1}P(f+x\alpha)$ are both elements of $\C[S][\alpha]$, hence
\[\C[V]=\C[S][\alpha].\]

\medskip
  
{\bf Step 2.} Note that  the image $D(f)\in\C[S]$ of $f$ under the derivation $D$ is divisible by $x$. Therefore, we can extend $D$ to a locally nilpotent derivation $\tilde{D}$ on $\C[V]=\C[S][\alpha]$ by letting 
\[\tilde{D}(\alpha)=-\frac{D(f)}{x}.\]
With this choice, we  then have that $\tilde{D}(z)=\tilde{D}(f+x\alpha)=0$ and that $\tilde{D}(y)=\tilde{D}(\frac{P(z)}{x})=0$. 

\medskip
{\bf Step 3.} To find a slice $s$ for $\tilde{D}$, it suffices to take a polynomial $g(x,t)\in\C[x,t]$ such that the congruences
\begin{equation}\tag{$\star\star$}\label{equa**}
g(x,r_i+xt)\equiv \sigma_i(x) \mod(x^n)
\end{equation}
hold in $\C[x,t]$  for all $1\leq i\leq d$, and to define \[s=\frac{u-g(x,z)}{x^n}.\]
Indeed, since every restriction 
\[(u-g(x,z))|_{\Vcal_i}=x^nu_i+\sigma_i(x)-g(x,r_i+xz_i)\] 
is divisible by $x^n$ in $\C[\Vcal_i]$, it follows that $s$ is a regular function on $V$. Moreover, it is clear that  $\tilde{D}(s)=1$.

In order to construct a suitable polynomial $g(x,t)$, one can proceed as follows. If we denote  
\[\sigma_i(x)=\sum_{j=0}^{n-1}a_{ij}x^j\quad \text{ with } a_{ij}\in\C,\]
then we can define 
\[g(x,t)=\sum_{j=0}^{n-1}g_j(t)x^j,\]
where the $g_j(t)\in\C[t]$ are Hermite interpolation polynomials such that
\[g_j(r_i)=a_{ij}\quad\text{ and }\quad g_j^{(k)}(r_i)=0\]
for all $1\leq i\leq d$ and all $0\leq j\leq n-1$, $1\leq k\leq n-1-j$.

\medskip
{\bf The isomorphism.} Finally, the above three steps have produced the  desired isomorphism. We have therefore proven the following result.
\begin{theorem}\label{main-thm} Let $S$ be a special Danielewski surface over $Z(d)$, and let $W$ be the hypersurface defined by the equation $XY=P(Z)$, where $P$ is a polynomial of degree $d$ whose roots are all simple. Suppose $f$ and $g$ are chosen to satisfy $($\ref{equa*}$)$ and $($\ref{equa**}$)$ above. Then the map 
\[\Phi\colon\C[W\times\A^1]=\C[X,Y,Z,W]/(XY-P(Z))=\C[x,y,z,w]\xrightarrow{\sim}\C[S\times\A^1]=\C[S][\alpha]\]
defined by
\begin{align*}
\Phi(x)&=x\\
\Phi(z)&=f+x\alpha\\
\Phi(y)&=\frac{P(f+x\alpha)}{x}\\
\Phi(w)&=\frac{u-g(x,f+x\alpha)}{x^n}
\end{align*}
is an isomorphism.
\end{theorem}

\begin{corollary}
Keeping the same notation as in the previous theorem, it follows that the special Danielewski surface $S$ is isomorphic to the surface defined by the equations 
\[xy=P(z) \text{ and } \Phi^{-1}(\alpha)=\lambda\]
in $\A^4$, where $\lambda\in\C$ is any constant.
\end{corollary}

\section{Some explicit examples}


\subsection{Russell's isomorphism}\hfill

In \cite{SY}, the authors give an explicit isomorphism, which is due to Russell, between the cylinders over the Danielewski surfaces of respective equations $xy=z^2-1$ and $x^2y=z^2-1$. See also Theorem 10.1 in \cite{Fre}. With our method, we can recover this isomorphism easily.  

In this section we will show how to apply the method of the previous section to treat  a slightly more general  case, and, in the end of the section, we will specialize to the case of the Russell isomorphism. 
We shall use the following notations. We denote by $P(z)=\prod_{i=1}^d(z-r_i)\in\C[z]$  a polynomial  with  $d\geq2$ simple roots and by $W_{n,P}$ the hypersurface in $\A^3$ defined  by the equation $x^ny=P(z)$, where $n\geq1$ is a positive integer. Moreover, we let 
\[\C[W_{n,P}]=\C[X,Y,Z]/(X^nY-P(Z))=\C[x_n,y_n,z_n]\]
and
\[\C[W_{n,P}\times\A^1]=\C[X,Y,Z,W]/(X^nY-P(Z))=\C[x_n,y_n,z_n,w_n],\]
 where $(x_n)^ny_n=P(z_n)$.

Accordingly with the previous section, we now proceed to construct an isomorphism between $\C[W_{1,P}\times\A^1]$ and $\C[W_{2,P}\times\A^1]$.

First, note that the regular function $f=z_2\in\C[W_{2,P}]$ is equal to $r_i$ on every point of the line $\{x=0, z=r_i\}\subset W_{2,P}$. Also, in this case, $u=z_2\in\C[W_{2,P}]$  restricts to a coordinate function on every general fiber of the projection $\textrm{pr}_{x}\colon W_{2,P}\to\C$. 

Secondly, since $P$ has only simple roots, there exist two polynomials $U,V\in\C[z]$ such that $U(z)P'(z)+V(z)P(z)=1$ in $\C[z]$. Then, the polynomial  
\[g(z)=z-P(z)U(z)\in\C[z]\]
satisfies that  \[g(r_i)=r_i\] 
and \[g'(r_i)=1-P'(r_i)U(r_i)-P(r_i)U'(r_i)=0\]
for all $1\leq i\leq d$. 
With these choices for $f$, $u$ and $g$,  we get the isomorphism
\[\Phi\colon\C[W_{1,P}\times\A^1]=\C[x_1,y_1,z_1,w_1]\xrightarrow{\sim}\C[W_{2,P}\times\A^1]=\C[x_2,y_2,z_2,w_2]\]
defined by
\begin{align*}
\Phi(x_1)&=x_2\\
\Phi(z_1)&=z_2+x_2w_2\\
\Phi(y_1)&=\frac{P(z_2+x_2w_2)}{x_2}\\
\Phi(w_1)&=\frac{z_2-g(z_2+x_2w_2)}{x_2^2},
\end{align*}
whose inverse isomorphism 
\[\Psi\colon\C[W_{2,P}\times\A^1]=\C[x_2,y_2,z_2,w_2]\xrightarrow{\sim}\C[W_{1,P}\times\A^1]=\C[x_1,y_1,z_1,w_1]\]
is defined by
\begin{align*}
\Psi(x_2)&=x_1\\
\Psi(z_2)&=x_1^2w_1+g(z_1)\\
\Psi(y_2)&=\frac{P(x_1^2w_1+g(z_1))}{x_1^2}\\
\Psi(w_2)&=\frac{z_1-(x_1^2w_1+g(z_1))}{x_1}.
\end{align*}

In the special case where $P(z)=z^2-1$, we have $g(z)=z-(z^2-1)\dfrac{z}{2}$ and we thus   obtain the  inverse isomorphisms
\[\Phi_*\colon \{x^2y=z^2-1\}\times\A^1\to \{xy=z^2-1\}\times\A^1\]
and 
\[\Psi_*\colon \{xy=z^2-1\}\times\A^1\to \{x^2y=z^2-1\}\times\A^1\]
defined by
\begin{align*}
\Phi_*(x,y,z,w)&=\Big(x,\frac{P(z+xw)}{x},z+xw,\frac{z-g(z+xw)}{x^2}\Big)\\
&=\Big(x,\frac{(z+xw)^2-1}{x},z+xw,\\
&\qquad\qquad\frac{z-(z+xw)+\frac{1}{2}(z+xw)((z+xw)^2-1)}{x^2}\Big)\\
&=\Big(x,\frac{z^2-1}{x}+2zw+xw^2,z+xw,\\
&\qquad\qquad\frac{\frac{1}{2}(z+xw)(z^2-1+x^2w^2)+xw(z^2-1)+x^2w^2z}{x^2}\Big)\\
&=\Big(x,xy+2zw+xw^2,z+xw,\frac{1}{2}(z+xw)(y+w^2)+xyw+w^2z\Big)\\
&=\Big(x,xy+2zw+xw^2,z+xw,\frac{1}{2}(yz+3zw^2+3xyw+xw^3)\Big)
\end{align*}
and 
\begin{align*}
\Psi_*(x,y,z,w)&=\Big(x,\frac{P(x^2w+g(z))}{x^2},x^2w+g(z),\frac{z-x^2w-g(z)}{x}\Big)\\
&=\Big(x,\frac{x^4w^2+2x^2wg(z)+(z^2-1)^2(\frac{1}{4}z^2-1)}{x^2},x^2w+g(z),\\
&\qquad\qquad-xw+\frac{1}{2}\cdot\frac{z(z^2-1)}{x}\Big)\\
&=\Big(x,x^2w^2+2wg(z)+y^2(\frac{1}{4}z^2-1),x^2w+g(z),-xw+\frac{1}{2}zy\Big).
\end{align*}

\subsection{Classical Danielewski surfaces}

In light of the previous example,  we obtain simple explicit isomorphisms between the cylinders over two classical Danielewski surfaces. 

\begin{proposition}\label{prop-classical-DS}
Let $d,n,m\geq1$ be positive integers and let $P(z)=\prod_{i=1}^d(z-a_i)$ and $Q(z)=\prod_{i=1}^d(z-b_i)$ be polynomials in $\C[z]$ with simple roots. Recall that $W_{n,P}$ and $W_{m,Q}$ denote the hypersurfaces in $\A^3=\Spec(\C[x,y,z])$ that are defined respectively by the equation 
\[W_{n,P}\colon x^ny=P(z)\]
and
\[W_{m,Q}\colon x^my=Q(z).\] 

Let $f,g\in\C[z]$ be two Hermite interpolating polynomials such that 
\[f(b_i)=a_i\quad\text{ and }\quad f^{(k)}(b_i)=0\quad \text{ for all } 1\leq i\leq d, 1\leq k\leq n-1\]
and   
\[g(a_i)=b_i\quad\text{ and }\quad g^{(k)}(a_i)=0\quad \text{ for all } 1\leq i\leq d, 1\leq k\leq m-1.\]
Then, the maps  
\begin{align*}
\varphi\colon &W_{n,P}\times\A^1\to W_{m,Q}\times\A^1\\
&(x,y,z,w)\mapsto(x,\frac{Q(g(z)+x^mw)}{x^m},g(z)+x^mw,\frac{z-f(g(z)+x^mw)}{x^n})
\end{align*}
and 
\begin{align*}
\psi\colon &W_{m,Q}\times\A^1\to W_{n,P}\times\A^1\\
&(x,y,z,w)\mapsto(x,\frac{P(f(z)+x^nw)}{x^n},f(z)+x^nw,\frac{z-g(f(z)+x^nw)}{x^m})
\end{align*}
are regular and define inverse isomorphisms between the cylinders $W_{n,P}\times\A^1$ and $W_{m,Q}\times\A^1$.    
\end{proposition}

\begin{proof}
On the one hand, we have that 
\[(P\circ f)(b_i)=(P\circ f)'(b_i)=\cdots=(P\circ f)^{(n-1)}(b_i)=0\quad\text{ for all } i.\]
This shows that $P(f(z))$ is divisible by $(Q(z))^n$, hence that $P(f(z))/x^n$ is a regular function on $W_{m,Q}$. Similarly,  $Q(g(z))/x^m$ is a regular function on $W_{n,P}$. 

On the other hand, we have that 
\[z-g(f(z)+x^nw)=z-g(f(z))-\sum_{k=1}^{\infty}\frac{(x^nw)^k}{k!}g^{(k)}(f(z))\]
is an element of the ideal $(Q(z),x^m)\C[x,z,w]$. Therefore, $x^{-m}(z-g(f(z)+x^nw))$ is a regular function on $W_{m,Q}\times\A^1=\Spec(\C[x,y,z,w]/(x^my-Q(z)))$. Similarly,  $x^{-n}(z-f(g(z)+x^mw))$ is a regular function on $W_{n,P}\times\A^1$. 

Thus, $\varphi$ and $\psi$ are regular maps. It is moreover straightforward to check that they are inverse of each other.
\end{proof}

\subsection{An iterated Danielewski hypersurface} 

Let us look again at the iterated Danielewski hypersurface  $H=\{xz=(xy+z^2)^2-1\}$ in $\A^3$ that we studied at Section \ref{Section:construction-iterated-Danielewski}. Recall that the special fiber consists of the four lines   
\[\ell_i=\{(0,y,\varepsilon^i)\mid y\in\C\}, 1\leq i\leq 4,\]
where $\varepsilon=\boldsymbol{i}\in\C$ denotes a primitive fourth root of the unity, and that the surface $H$ corresponds to the data $n=2$ and $\sigma_i(x)=\varepsilon^{2i}+\frac{\varepsilon^{-i}}{2}x$ for all $1\leq i\leq 4$.

We  give now an isomorphism from $H\times\A^1$ to $\{xy=z^4-1\}\times\A^1$. Keeping the notations of Section~\ref{Section:plan of the construction}, we obtain the isomorphism
\begin{align*}
\{xz=(xy+z^2)^2-1\}\times\A^1&\to \{xy=z^4-1\}\times\A^1\\
(x,y,z,w)&\mapsto(x,\frac{(f+xw)^4-1}{x},f+xw,\frac{u-g(x,f+xw)}{x^2}),
\end{align*}
where  
\begin{align*}
r_i&=\varepsilon^i\\  
f&=z\\
u&=xy+z^2\\
g(x,z)&=z^2-\frac{1}{2}z^2(z^4-1)+x\frac{1}{2}z^3.
\end{align*}

\subsection{A double Danielewski surface}
We consider again  the double Danielewski surface  
\[D= \{xy=z^2-1 \text{ and } xt=y^2-1\}\quad \text{ in } \A^4\] that we described at Section \ref{Section:construction-double-Danielewski}. Recall that the special fiber consists of the four lines  
\[\ell_{i j}=\{(0,i,j,t)\mid t\in\C\}\}\] 
and that the surface $D$ corresponds to the data $n=2$ and $\sigma_{i j}(x)=j+\frac{ij}{2}x$, where $(i,j)\in\{1,-1\}\times\{1,-1\}$. 

Following the notations of Section~\ref{Section:plan of the construction}, we denote by  $\varepsilon=\boldsymbol{i}\in\C$  a primitive fourth root of  unity and define 
\begin{align*}
f&=\frac{z+y}{2}+\varepsilon\frac{y-z}{2}\\
u&=z\\
g(x,z)&=\frac{1-\varepsilon}{2}z^3+\frac{1+\varepsilon}{2}z-(z^4-1)\frac{z}{4}(3\frac{1-\varepsilon}{2}z^2+\frac{1+\varepsilon}{2})+x\frac{1}{2}z^2.
\end{align*}
Then, we have that 
\begin{align*}
&f|_{\ell_{11}}=1,  && g(x,1)=\sigma_{1 1}(x)\\
&f|_{\ell_{1 -1}}=\varepsilon,  && g(x,\varepsilon)=\sigma_{1 -1}(x)\\
&f|_{\ell_{-1 1}}=-\varepsilon,  && g(x,-\varepsilon)=\sigma_{-1 1}(x)\\
&f|_{\ell_{-1 -1}}=-1,   && g(x,-1)=\sigma_{-1 -1}(x)
\end{align*} 
and 
\[\frac{\partial g}{\partial z}(x,\varepsilon^i)\equiv0 \mod (x)\]
for all $1\leq i\leq 4$. This produces the isomorphism
\begin{align*}
D\times\A^1&\to \{xy=z^4-1\}\times\A^1\\
(x,y,z,t,w)&\mapsto(x,\frac{(f+xw)^4-1}{x},f+xw,\frac{u-g(x,f+xw)}{x^2}).
\end{align*}


\end{document}